\documentclass[12pt,a4paper]{article}
\usepackage{authblk}
\usepackage{amssymb}

\usepackage{enumerate}
\usepackage[all]{xy}
\usepackage{amsmath,amsthm,amssymb,amscd}
\newtheorem{thm}{Theorem}
\newtheorem{prop}{Proposition}
\newtheorem{cor}{Corollary}

\newtheorem{lem}{Lemma}

\newcommand{\bt}{\begin{thm}}
\newcommand{\et}{\end{thm}}
\newcommand{\bl}{\begin{lem}}
\newcommand{\el}{\end{lem}}
\newcommand{\bp}{\begin{prop}}
\newcommand{\ep}{\end{prop}}
\newcommand{\bc}{\begin{cor}}
\newcommand{\ec}{\end{cor}}

\newcommand{\biindice}[3]%
{

\begin{array}[t]{c}
#1\\
{\scriptstyle #2}\\
{\scriptstyle #3}
\end{array}

}

\begin{document}
%========================================================
%========================================================
\baselineskip7mm
%========================================================
%========================================================
%========================================================
%========================================================
\title{\large \textbf{Structure of a factor ring $R/P$ in terms of differential identities involving  a new kind of  involution}}
 %\centerline{}
 \author[1]{ Karim Bouchannafa \thanks{Corresponding author: bouchannafa.k@gmail.com}}
 \author[1]{ Lahcen Oukhtite}
 \author[1]{ Mohammed Zerra}

\affil[1]{\small Department of Mathematics, Faculty of Sciences and Technology, University S. M. Ben Abdellah, Fez, Morocco}
\affil[ ]{\small bouchannafa.k@gmail.com,  mohamed.zerra@gmail.com and, oukhtitel@hotmail.com }

\date{ }
\maketitle

\maketitle
\abstract {Let $R$ be a ring and $P$ a prime ideal of $R.$ In this paper, we establish some commutativity criteria for the factor ring $R/P$ in terms of derivations of $R$ satisfying some algebraic identities involving  a new kind of involution in relation with $P.$}
\\
\vskip 10pt
\emph{Key words:} {\small Prime ideal, factor ring, derivation, commutativity, involution.}

\vskip 10pt
\emph{2010 Mathematics Subject Classification:} 16N60; 46J10; 16W25; 16N20.
%\linenumbers

\section*{1. Introduction}
\hspace{4 ex} All along this article, $R$ will  represent an associative ring with center $Z(R)$. For any $a,b\in R$, the symbol $[a,b]$ will denote the commutator $ab - ba$, while the symbol $a\circ b$ will stand for the anti-commutator $ab+ba$. Recall that a proper ideal $P$ of $R$ is said to be prime if for any $a,b\in R$, $aRb\subseteq P$ implies that $a\in P$ or $b\in P$. Furthermore, $R$ is called a prime ring if and only if $(0)$ is the prime ideal of $R$. The ring $R$ is called a semi-prime if for any $a\in R$, $aRa = (0)$ implies that $a = 0$, and $R$ is said to be 2-torsion free if $2a = 0$ implies $a = 0$, for each $a \in R$.\\
\hspace*{4 ex} An involution of $R$ is an anti-automorphism of order $2$, that is an additive mapping $*:R\rightarrow R$ such that $(a^*)^* = a$ and $(ab)^* = b^*a^*$, for any $a,b\in  R$. A ring equipped with an involution is known as a $*$-ring. An element $a$ in a $*$-ring is said to be hermitian if $a^*= a$ and skew-hermitian if $a^*=-a$. The sets of all hermitian and skew-hermitian elements of $R$ will be denoted by $H(R)$ and $S(R)$ respectively, the involution is said to be of the first kind if $Z(R)\subseteq H(R)$, otherwise it is said to be of the second kind. In the latter, case $Z(R)\cap S(R)\neq \{0\}$.\\
\hspace*{4 ex} An additive map $d : R \rightarrow R$ is called a derivation on $R$ if $d(xy) = d(x)y + xd(y)$ for all $x,y\in R$, let $a\in R$ be a fixed element, the map $d : R \rightarrow R$ defined by $d(x) =[a, x]$, for all $x\in R$, is a derivation on $R$ called the inner derivation induced by $a$. Many authors have established the commutativity of prime and semi-prime rings admitting suitably constrained additive mappings such as derivations, generalized derivations and endomorphisms acting on appropriate subsets of the rings (see \cite{AD,ADA,MNO,ZERRA3,HMO} for details). The most famous result in this area is due to Herstein \cite{Herst} who proved that if a prime ring $R$ of characteristic different from two admits a nonzero derivation $d$ such that $d(x)d(y) = d(y)d(x)$ for all $x,y \in R$, then $R$ is commutative. This remarkable theorem of Herstein has been influential and it has played a key role in the development of various notions. Motivated by this result, Bell and Daif \cite{BD} obtained the same result by considering the identity $d[x,y] = 0$ for all $x$ and $y$ in a nonzero ideal of $R$. Later, Ali et al. \cite{ADA} showed that if a prime ring $R$ with involution $*$ of a characteristic different from $2$ admits a nonzero derivation $d$ such that $d(xx^*)=d(x^*x)$ for all $x\in R$ and $Z(R)\cap S(R)\neq \{0\}$, then $R$ is commutative. In 2013, Oukhtite et al. \cite{OMA} studied the following identities for $*$-prime ring $R$: $(i)\,d[x,y]=0$, $(ii)\, d[x,y]\pm[x,y]\in Z(R)$, $(iii)\,d(x\circ y)=0$ and $(iv)\,d(x\circ y)\pm x\circ y\in Z(R)$ for all $x,y\in J$, where $J$ is a nonzero Jordan ideal of $R$.

The present paper is motivated by the previous results and we here continue this line of investigation by considering a generalization to an arbitrary ring rather than a prime ring. More precisely, we will establish a relationship between the structure of a factor ring  $R/P$  and the behavior of its derivations satisfying some identities with central values.

\section*{2. Main results}

%=================================
Let $R$ be any ring with involution $*$ and $P$ is a prime ideal of $R$. Motivated by various results in the literature, our purpose is to study a pair of derivations $(d_{1},d_{2})$ satisfying any one of the following properties:\\

\noindent{\rm (1)} $ \overline{[d_1(x),x^*]\pm [x,d_2(x^*)]}\in Z(R/P)\;\;\mbox{for all}\;\;x\in R$,\\
{\rm (2)} $\overline{d_1(x)\circ x^*\pm x\circ d_2(x^*)}\in Z(R/P)\;\;\mbox{for all}\;\;x\in R$,\\
{\rm (3)} $\overline{d_1(x)d_2(x^*)}\in Z(R/P)\;\;\mbox{for all}\;\; x\in R$,\\
{\rm (4)} $\overline{d_1(x)d_2(x^*)\pm [x,x^*]}\in Z(R/P)$ for all $x\in R$,\\
{\rm (5)} $\overline{d_1(x)d_2(x^*)\pm x\circ x^*}\in Z(R/P)\;\;\mbox{for all}\;\;x\in R$,\\
{\rm (6)} $\overline{d_1(x^*x)\pm d_2(xx^*)\pm xx^*}\in Z(R/P)\;\;\mbox{for all }\;\;x\in R$.\\

\noindent More precisely, we will discuss the existence of such mappings and their relationship with commutativity of the factor ring $R/P$. We begin with the following notions.\\

\noindent \textbf{Definition 1.} Let $R$ be a ring with involution $*$ and $I$ an ideal of $R$. The involution $*$ is said to be $I$-first kind on $R$ if $Z(R)\cap S(R)\subseteq I$. Otherwise, it is said to be of the $I$-second kind.\\

\noindent
We begin our discussions with the following observations:\\

\noindent
{\bf Remark 1. }\\
(1) A $(0)$-first kind involution is of the first kind. The converse is true if the ring is 2-torsion free. \\
(2) Let $I$ be an ideal of a ring $R$ of a 2-torsion free. An involution of the first kind over $R$ is $I$-first kind.\\

\noindent
The following counter-example  shows that the converse of (2) is not necessary true.\\

\noindent
{\bf Counter-example 1.}\\ Let us consider the ring $R=\left\{\begin{pmatrix}
0 & 0 & 0\\
z_{1} & 0 & 0\\
z_{2} & z_{1} & 0
\end{pmatrix}\;\bigg|\; z_{1},z_{2}\in \mathbb{C}\right\}$ provided with the involution of the second kind $\ast$ defined by
$
\left(
\begin{array}{ccc}
0 & 0 & 0\\
z_{1} & 0 & 0\\
z_{2} & z_{1} & 0
\end{array}
\right)^{\ast}= \left(
\begin{array}{ccc}
0 & 0 & 0\\
\overline{z_{1}} & 0 & 0\\
\overline{z_{2}} & \overline{z_{1}} & 0
\end{array}
\right).$ \\
It is clear that $Z(R) \cap S(R)=\left\{\begin{pmatrix}
0 & 0 & 0\\
0 & 0 & 0\\
ir & 0 & 0
\end{pmatrix}\;\bigg|\; r\in \mathbb{R}\right\}.$ Furthermore, if we set\\
 $I=\left\{\begin{pmatrix}
0 & 0 & 0\\
0 & 0 & 0\\
z & 0 & 0
\end{pmatrix}\;\bigg|\; z\in \mathbb{C}\right\},$ then $I$ is an ideal of $R$ with
$Z(R)\cap S(R)\subseteq I$. Accordingly, $\ast$ is $I$-first kind.\\

\noindent
The following lemma is essential for developing the proofs of our results.

\begin{lem} {\rm (\cite{SA}, Lemma 2.2)} \label{lem}
Let $(R,*)$ be a ring with involution, $P$ a prime ideal of $R$ such that $char(R/P)\neq 2$ and  $Z(R)\cap S(R)\nsubseteq P$. If $[x,x^*]\in P$ for all $x\in R$, then $R/P$ is an integral domain.
\end{lem}

In \cite{AD}, it is proved that a $2$-torsion free prime ring with involution $(R,*)$ is necessary commutative, if it admits a nonzero derivation $d$ such that $d([x,x^{*}])=0$ for all $x \in R$, where $S(R) \cap Z(R)\neq \{0\}$. Motivated by this result we investigate a more general context of differential identities involving two derivations by omitting the primeness assumption imposed on the ring.

%=================================

\begin{thm}\label{thm1}
Let $(R,\ast)$ be a ring with involution of $P$-second kind where $P$ is a prime ideal of $R$ such that $char(R/P)\neq 2$. If $R$ admits two derivations $d_1$ and $d_2$ such that $$
\overline{[d_1(x),x^*]+[x,d_2(x^*)]}\in Z(R/P)\;\;\mbox{for all}\;\;x\in R,
$$
then  $R/P$ is an integral domain or \big($d_1(R)\subseteq P\;$ and $\;d_2(R)\subseteq P$\big).

\end{thm}

\begin{proof}
We are given that
\begin{equation}\label{eq1}
\overline{[d_1(x),x^*]+[x,d_2(x^*)]}\in Z(R/P)~~\mbox{for all}~~~~x\in R.
\end{equation}
A  linearization of (\ref{eq1}) gives
\begin{equation*}
\overline{[d_1(x),y^*]+[d_1(y),x^*]+[x,d_2(y^*)]+[y,d_2(x^*)]}\in Z(R/P)~~\mbox{for all}~~~~x,y\in R.
\end{equation*}
Replacing $y$ by $y^*$ in the above expression, we obtain
\begin{equation}\label{eq2}
\overline{[d_1(x),y]+[x,d_2(y)]+[d_1(y^*),x^*]+[y^*,d_2(x^*)]}\in Z(R/P)~~\mbox{for all}~~~~x,y\in R.
\end{equation}
Substituting $yh$ for $y$ in (\ref{eq2}), where $h\in H(R)\cap Z(R)\setminus\{0\}$, we find that
\begin{equation}\label{eq3}
\overline{[x,y]d_2(h)+[y^*,x^*]d_1(h)}\in Z(R/P)~~\mbox{for all}~~~~x,y\in R.
\end{equation}
As a special case of (\ref{eq3}), when we put $h=k^2$, where $k\in S(R)\cap Z(R)$, we arrive at
\begin{equation}\label{eq4}
\overline{[x,y]d_2(k)+[y^*,x^*]d_1(k)}\in Z(R/P)~~\mbox{for all}~~~~x,y\in R.
\end{equation}
If we put $y=yk$ in (\ref{eq4}), then we may write
\begin{equation}\label{eq5}
\overline{[x,y]d_2(k)-[y^*,x^*]d_1(k)}\in Z(R/P)~~\mbox{for all}~~~~x,y\in R.
\end{equation}
Replacing $yk$ instead of $y$ in (\ref{eq2}) and using (\ref{eq5}), it follows that
\begin{equation}\label{eq6}
\overline{[d_1(x),y]k+[x,d_2(y)]k-[d_1(y^*),x^*]k-[y^*,d_2(x^*)]k}\in Z(R/P)~~\mbox{for all}~~~~x,y\in R.
\end{equation}
Since $*$ is of $P$-second kind, then applying (\ref{eq2}) together with (\ref{eq6}), we arrive at
\begin{equation*}
2(\overline{[d_1(x),y]+[x,d_2(y)]})\in Z(R/P)~~\mbox{for all}~~~~x,y\in R.
\end{equation*}
The $2$-torsion freeness hypothesis yields
\begin{equation*}
\overline{[d_1(x),y]+[x,d_2(y)]}\in Z(R/P)~~\mbox{for all}~~~~x,y\in R.
\end{equation*}
Invoking (\cite{Hajar}, Theorem 2.6), we get the required result.
\end{proof}

As a consequence of Theorem \ref{thm1}, we obtain the following result.

\begin{cor}\label{cor1}
Let $(R,\ast)$ be a ring with involution of $P$-second kind where $P$ is a prime ideal of $R$ such that $char(R/P)\neq 2$. If $R$ admits a derivation $d$ such that
 $\;\;\overline{d[x,x^*]}\in Z(R/P)$ for all $x\in R,$ then
$d(R)\subseteq P\;$  or  $\;R/P$ is an integral domain.
\end{cor}

Letting $R$ be a prime ring, then $(0)$ is a prime ideal. In this case we get an improved version of Ali et al. \cite{ADA}.

\begin{cor}\label{cor2}
Let $R$ be a 2-torsion free prime ring with involution $*$ of the second kind, $d_1$ and $d_2$ are nonzero derivations of $R$, then $[d_1(x),x^*]\pm [x,d_2(x^*)]\in Z(R)$ for all $x\in R$, if and only if $R$ is an integral domain.
\end{cor}

\begin{cor}\label{cor3}
Let $R$ be a 2-torsion free prime ring with involution $*$ of the second kind, $F$ be a generalized derivation of $R$ associated with a nonzero derivation $d$, then $F(x\circ x^*)\pm d[x,x^*]\in Z(R)$ for all $x\in R$, if and only if $R$ is an integral domain.
\end{cor}

%================================

Inspired by the preceding results, one many ask: what happens if we consider the Jordan product instead of the Lie product in the previous theorem? The following theorem treats this situation.
\begin{thm}\label{thm2}
Let $(R,\ast)$ be a ring with involution of $P$-second kind where $P$ is a prime ideal of $R$, such that $char(R/P)\neq 2$. If $R$ admits two derivations $d_1$ and $d_2$ such that
$$
\overline{d_1(x)\circ x^*+ x\circ d_2(x^*)}\in Z(R/P)\;\;\mbox{for all}\;\;x\in R,
$$
then  $R/P$ is an integral domain or \big($d_1(R)\subseteq P\;$ and $\;d_2(R)\subseteq P$\big).

\end{thm}

\begin{proof}
Assume that
\begin{equation}\label{eqq1}
\overline{d_1(x)\circ x^*+ x\circ d_2(x^*)}\in Z(R/P)~~\mbox{for all}~~~~x\in R.
\end{equation}
Analogously, linearizing (\ref{eqq1}), we get
\begin{equation*}
\overline{d_1(x)\circ y^*+ d_1(y)\circ x^*+ x\circ d_2(y^*)+ y\circ d_2(x^*)}\in Z(R/P)~~\mbox{for all}~~~~x,y\in R,
\end{equation*}
in such a way that
\begin{equation}\label{eqq2}
\overline{d_1(x)\circ y+ x\circ d_2(y)+ d_1(y^*)\circ x^*+ y^*\circ d_2(x^*)}\in Z(R/P)~~\mbox{for all}~~~~x,y\in R.
\end{equation}
Replacing $y$ by $yh$ in (\ref{eqq2}), where $h\in Z(R)\cap H(R)$ and using it, we obtain
\begin{equation}\label{eqq3}
\overline{(x\circ y)d_2(h)+(y^*\circ x^*)d_1(h)}\in Z(R/P)~~\mbox{for all}~~~~x,y\in R.
\end{equation}
As a special case of (\ref{eqq3}), when we put $y = k^2$, we may write
\begin{equation}\label{eqq4}
2(\overline{xd_2(h)+ x^*d_1(h)})\in Z(R/P)~~\mbox{for all}~~~~x\in R.
\end{equation}
Substituting $xk$ for $x$ in (\ref{eqq4}), where $k\in Z(R)\cap S(R)$, one can see that
\begin{equation}\label{eqq5}
2(\overline{xd_2(h)- x^*d_1(h)})\in Z(R/P)~~\mbox{for all}~~~~x\in R.
\end{equation}
Adding relations (\ref{eqq4}) and (\ref{eqq5}) and using $2$-torsion freeness hypothesis, we arrive at
\begin{equation*}
\overline{x}\in Z(R/P)~~\mbox{for all}~~x\in R\quad \mbox{or}\quad d_2(h)\in P.
\end{equation*}
If $d_2(h)\in P$ for all $h\in Z(R)\cap S(R)$, then the relation (\ref{eqq4}) yields $\overline{2x^*d_1(h)}\in Z(R/P)$ for all $x\in R$. Therefore, $\overline{x^*}\in Z(R/P)$ for all $x\in R$ or $d_1(h)\in P$, and so we have also $d_1(k)\in P$ and $d_2(k)\in P$ for all $k\in Z(R)\cap S(R)$.\\
Setting $y=k$ in (\ref{eqq2}), where $k\in Z(R)\cap S(R)$, we may write
\begin{equation}\label{eqq6}
\overline{d_1(x)-d_2(x^*)}\in Z(R/P)~~\mbox{for all}~~~~x\in R.
\end{equation}
Writing $xk$ for $x$ in (\ref{eqq6}), where $k\in Z(R)\cap S(R)$, we find that
\begin{equation}\label{eqq7}
\overline{d_1(x)+d_2(x^*)}\in Z(R/P)~~\mbox{for all}~~~~x\in R.
\end{equation}
The summation of (\ref{eqq6}) and (\ref{eqq7}) gives $\overline{d_1(x)}\in Z(R/P)$ for all $x\in R$. Hence  (\ref{eqq7}) reduces to $\overline{d_2(x)}\in Z(R/P)$ for all $x\in R$. From (\cite{AMT}, Lemma 2.1) it follows that $d_1(R)\subseteq P$ and $d_2(R)\subseteq P$ or $R/P$ is commutative. Now, if $\overline{x^*}\in Z(R/P)$ for all $x\in R$, then $[R,R]\subseteq P$. Consequently, Lemma \ref{lem} completes our proof.
\end{proof}

In \cite{AD}, it is showed that if $R$ is a 2-torsion free prime ring with involution $*$ such that
$S(R)\cap Z(R)\neq \{0\}$ admitting a nonzero derivation $d$ satisfying $d(x\circ x^{*}) = 0$ for all $x\in R$, then $R$ is commutative. The author in \cite{Fuad}, proved that this result is not true. Here we consider the more general identity  $\overline{d(x\circ x^*)}\in Z(R/P).$
\begin{cor}\label{cor4}
Let $(R,\ast)$ be a ring with involution of $P$-second kind where $P$ is a prime ideal of $R$ such that $char(R/P)\neq 2$. If $R$ admits a derivation $d$ such that 
 $$
\overline{d(x\circ x^*)}\in Z(R/P)\;\;\mbox{for all}\;\;x\in R,
 $$
then $\;R/P\;$ is an integral domain or $\;d(R)\subseteq P.$
\end{cor}

As an application of Theorem \ref{thm2}, we establish some commutativity criteria for prime rings.

\begin{cor}\label{cor5}
Let $R$ be a 2-torsion free prime ring with involution $*$ of the second kind, $d_1$ and $d_2$ are nonzero derivations of $R$, then $d_1(x)\circ x^*\pm x\circ d_2(x^*)\in Z(R)$ for all $x\in R$, if and only if $R$ is an integral domain.
\end{cor}

\begin{cor}\label{cor6}
Let $R$ be a 2-torsion free prime ring with involution $*$ of the second kind. If $R$ admits a generalized derivation $F$ associated with a nonzero derivation $d$, then $F[x,x^*]\pm d(x\circ x^*)\in Z(R)$ for all $x\in R$, if and only if $R$ is an integral domain.
\end{cor}

\begin{thm}\label{thm3}
Let $(R,\ast)$ be a ring with involution of $P$-second kind where $P$ is a prime ideal of $R$, such that $char(R/P)\neq 2$. If $R$ admits two derivations $d_1$ and $d_2$ such that 
$$
\overline{d_1(x)d_2(x^*)}\in Z(R/P)\;\;\mbox{for all}\;\;x\in R,
$$
then one of the following assertions holds:\\
{\em (1)} $d_1(R)\subseteq P$;\\
{\em (2)} $d_2(R)\subseteq P$;\\
{\em (3)} $R/P$ is an integral domain.
\end{thm}

\begin{proof}
By assumption, we have
\begin{equation}\label{equation1}
\overline{d_1(x)d_2(x^*)}\in Z(R/P)~~\mbox{for all}~~~~x\in R.
\end{equation}
A linearization of (\ref{equation1}) yields
\begin{equation}\label{equation2}
\overline{d_1(x)d_2(y^*)+d_1(y)d_2(x^*)}\in Z(R/P)~~\mbox{for all}~~~~x,y\in R.
\end{equation}
Replacing $y$ by $yh$ in (\ref{equation2}), where $h\in Z(R)\cap H(R)\setminus \{0\}$, it follows that
\begin{equation}\label{equation3}
\overline{d_1(x)y^*d_2(h)+d_1(h)yd_2(x^*)}\in Z(R/P)~~\mbox{for all}~~~~x,y\in R.
\end{equation}
Writing $k^2$ instead of $h$ in (\ref{equation3}), where $h\in Z(R)\cap S(R)\setminus \{0\}$ and employing it, we can see that
\begin{equation}\label{equation4}
\overline{d_1(x)y^*d_2(k)+d_1(k)yd_2(x^*)}\in Z(R/P)~~\mbox{for all}~~~~x,y\in R.
\end{equation}
Substituting $yk$ for $yk$ in (\ref{equation4}) and applying the hypothesis, we arrive at
\begin{equation}\label{equation5}
\overline{-d_1(x)y^*d_2(k)+d_1(k)yd_2(x^*)}\in Z(R/P)~~\mbox{for all}~~~~x,y\in R.
\end{equation}
Putting $yk$ instead of $y$ in (\ref{equation2}), we are forced to conclude that
\begin{equation}\label{equation6}
\overline{-d_1(x)d_2(y^*)k-d_1(x)y^*d_2(k)+d_1(y)d_2(x^*)k+d_1(k)yd_2(x^*)}\in Z(R/P).
\end{equation}
Invoking (\ref{equation5}) together with (\ref{equation6}), we find that
\begin{equation}\label{equation7}
\overline{-d_1(x)d_2(y^*)+d_1(y)d_2(x^*)}\in Z(R/P)~~\mbox{for all}~~~~x,y\in R.
\end{equation}
Subtracting (\ref{equation2}) from (\ref{equation7}) and employing the $2$-torsion freeness hypothesis, we easily get
\begin{equation}\label{equation8}
\overline{d_1(x)d_2(y)}\in Z(R/P)~~\mbox{for all}~~~~x,y\in R.
\end{equation}
Writing $yr$ instead of $y$ in (\ref{equation8}), it is obvious to see that
\begin{equation*}
[d_1(x)yd_2(r),r]\in P~~\mbox{for all}~~~~r,x,y\in R.
\end{equation*}
Replacing $x$ by $xt$ in the above equation and using it, we arrive at
\begin{equation}\label{equation9}
[x,r]d_1(t)yd_2(r)\in P~~\mbox{for all}~~~~r,t,x,y\in R.
\end{equation}
Substituting $sx$ for $x$ in (\ref{equation9}) and using it, one can verify that
\begin{equation}
[s,r]Rd_1(t)Rd_2(r)\subseteq P~~\mbox{for all}~~~~r,s,t\in R.
\end{equation}
Using the primeness of $P$ together with Brauer's trick, we obtain that $R/P$ is an integral domain or $d_1(R)\subseteq P$ or $d_2(R)\subseteq P$. This completes our proof.
\end{proof}

\noindent
Now if we set $d=d_{1}=d_{2}$, then we get the following result.

\begin{cor}\label{cor7}
Let $(R,\ast)$ be a ring with involution of $P$-second kind where $P$ is a prime ideal of $R$ such that $char(R/P)\neq 2$. If $R$ admits a derivation $d$ such that $$
\overline{d(x)d(x^*)}\in Z(R/P)\;\;\mbox{for all}\;\;x\in R,
$$ then  $\;d(R)\subseteq P\;$ or $\;R/P\;$ is an integral domain.
\end{cor}

\begin{cor}\label{cor8}
Let $R$ be a $2$-torsion free prime ring with involution $*$ of the second kind, $d_1$ and $d_2$ are nonzero derivations of $R$, then $d_1(x)d_2(x^*)\in Z(R)$ for all $x\in R$, if and only if $R$ is an integral domain.
\end{cor}

%================================
In (\cite{Nejjar1}, Theorems 3.5 and 3.8), it is demonstrated that if $(R,*)$ is a 2-torsion free prime ring with involution provided with a derivation $d$ which satisfies $d(x)\circ d(x^{*})\pm x\circ x^{*}\in Z(R)$  for all $x \in R $  or $[d(x),d(x^{*})]\pm x\circ x^{*}\in Z(R)$  for all $x \in R $, then $R$ must be commutative. This result has been generalized in  (\cite{ZERRA3}, Theorem 3)  by considering a pair of derivations $(d,g)$. Inspired by the previous results, the next theorem treats the following special identities:\\
$\overline{d_{1}(x)d_{2}(x^{*})- x\circ x^{*}}\in Z(R/P)$ for all $x\in R$,\\
$\overline{d_{1}(x)d_{2}(x^{*})-  [x,x^{*}]}\in Z(R/P)$ for all $x\in R.$

\begin{thm}\label{thm4}
Let $(R,\ast)$ be a ring with involution of $P$-second kind where $P$ is a prime ideal of $R$ such that $char(R/P)\neq 2$. If $R$ admits two derivations $d_1$ and $d_2,$ then the following assertions are equivalent:\\
{\rm (1)} $\overline{d_1(x)d_2(x^*)-[x,x^*]}\in Z(R/P)$ for all $x\in R$;\\
{\rm (2)} $\overline{d_1(x)d_2(x^*)-x\circ x^*}\in Z(R/P)$ for all $x\in R$;\\
{\rm (3)} $R/P$ is an integral domain.
\end{thm}

\begin{proof}
$(1) \implies (3)$ We are given that
\begin{equation}\label{equa1}
\overline{d_1(x)d_2(x^*)-[x,x^*]}\in Z(R/P)~~\mbox{for all}~~~~x\in R.
\end{equation}
Linearizing (\ref{equa1}), we obtain
\begin{equation*}
\overline{d_1(x)d_2(y^*)+d_1(y)d_2(x^*)-[x,y^*]-[y,x^*]}\in Z(R/P)~~\mbox{for all}~~~~x,y\in R.
\end{equation*}
Namely, we get that
\begin{equation}\label{equa2}
\overline{d_1(x)d_2(y)-[x,y]+d_1(y^*)d_2(x^*)-[y^*,x^*]}\in Z(R/P)~~\mbox{for all}~~~~x,y\in R.
\end{equation}
Substituting $yh$ for $y$ in (\ref{equa2}), where $h\in H(R)\cap Z(R)\setminus \{0\}$, we can see that
\begin{equation*}
\overline{d_1(x)y^*d_2(h)+yd_2(x^*)d_1(h)}\in Z(R/P)~~\mbox{for all}~~~~x,y\in R.
\end{equation*}
This can be rewritten as
\begin{equation}\label{equa3}
\overline{d_1(x)yd_2(h)+y^*d_2(x^*)d_1(h)}\in Z(R/P)~~\mbox{for all}~~~~x,y\in R.
\end{equation}
Replacing $y$ by $yk$ in (\ref{equa3}) and applying the hypothesis, we arrive at
\begin{equation}\label{equa4}
\overline{d_1(x)yd_2(h)-y^*d_2(x^*)d_1(h)}\in Z(R/P)~~\mbox{for all}~~~~x,y\in R.
\end{equation}
Combining (\ref{equa3}) with (\ref{equa4}), we obtain
\begin{equation*}
2\overline{d_1(x)yd_2(h)}\in Z(R/P)~~\mbox{for all}~~~~x,y\in R.
\end{equation*}
Letting $y=k$ in the above expression and using the primeness of $P$, we conclude that $\overline{d_1(x)}\in Z(R/P)$ for all $x\in R$ or $d_2(h)\in P$ for all $h\in H(R)\cap Z(R)$. In the later case, equation (\ref{equa3}) leads to $\overline{d_2(x)d_1(h)}\in Z(R/P)$ for all $x\in R$. In such a way that $\overline{d_1(x)}\in Z(R/P)$ for all $x\in R$ or $d_2(h)\in P$.\\
If $\overline{d_i(x)}\in Z(R/P)$ for all $x\in R$, where $i\in\{1,2\}$, then $[d_i(x),x]\in P$ for all $x\in R$. According to (\cite{AMT}, Lemma 2.1), we acquire that $d_i(R)\subseteq P$ or $R/P$ is an integral domain.\\
In the sequel we assume that $d_i(h)\in P$ for all $h\in Z(R)\cap H(R)$, where $i\in\{1,2\}$. Invoking arguments used as above, we can find that $d_i(k)\in P$ for all $k\in Z(R)\cap S(R)$. In this case, writing $yk$ for $y$ in (\ref{equa2}), we can show that
\begin{equation}\label{equa5}
\overline{d_1(x)d_2(y)-[x,y]-d_1(y^*)d_2(x^*)+[y^*,x^*]}\in Z(R/P)~~\mbox{for all}~~~~x,y\in R.
\end{equation}
Combining (\ref{equa2}) with (\ref{equa5}), we arrive at
\begin{equation}\label{equa6}
\overline{d_1(x)d_2(y)-[x,y]}\in Z(R/P)~~\mbox{for all}~~~~x,y\in R.
\end{equation}
Taking $y=x$ in (\ref{equa6}), we have
\begin{equation}\label{equa7}
\overline{d_1(x)d_2(x)}\in Z(R/P)\;\;\; \mbox{for all}\;\; x\in R.
\end{equation}
Again, replacing $y$ by $yx$ in (\ref{equa6}), we then get
\begin{equation}\label{equa8}
[d_1(x)yd_2(x),x]\in P\;\;\; \mbox{for all}\;\; x,y\in R.
\end{equation}
Substituting $d_2(x)y$ for $y$ in (\ref{equa8}), we find that
\begin{equation*}
d_1(x)d_2(x)[yd_2(x),x]\in P\;\;\; \mbox{for all}\;\; x,y\in R.
\end{equation*}
Hence, it follows that either $d_1(x)d_2(x)\in P$ or $[yd_2(x),x]\in P$ for all $x,y\in R$. If $(d_1.d_2)(R)\not\subseteq P$, then there exists $x_0\in R$ such that $d_1(x_0)d_2(x_0)\not\in P$. Therefore, for such $x_0$, we have $[yd_2(x_0),x_0]\in P$; in this case, putting $ry$ instead of $y$, we get $[r,x_0]yd_2(x_0)\in P$ for all $r,y\in R$. So that either $\overline{x_0}\in Z(R/P)$ or $d_2(x_0)\in P$. According to the second case, the main equation becomes $\overline{[x,x_0]}\in Z(R/P)$ for all $x\in R$; and thus $\overline{x_0}$ is necessary central in $R/P$.\\
In other words, if $\overline{x_0}\in Z(R/P)$, then in view of the above argument, we conclude that $[R,R]\subseteq P$, which implies that $R/P$ is commutative or $d_1(x_0)d_2(x_0)\in P$. Whence, for all $x\in R$, $d_1(x)d_2(x)\in P$. The linearization of this expression leads to
\begin{equation}\label{equa9}
d_1(x)d_2(y)+d_1(y)d_2(x)\in P\;\;\; \mbox{for all}\;\; x,y\in R.
\end{equation}
Writing $yx$ for $x$ in (\ref{equa9}), we see that
\begin{equation}\label{equa10}
d_1(x)yd_2(x)-d_1(y)[d_2(x),x]\in P\;\;\; \mbox{for all}\;\; x,y\in R.
\end{equation}
Substituting $ry$ for $y$ in (\ref{equa10}), we obtain by comparing this new result with (\ref{equa10}) that
\begin{equation}\label{equa22}
[r,d_1(x)]yd_2(x) + d_1(r)y[d_2(x),x]\in P\;\;\; \mbox{for all}\;\; r,x,y\in R.
\end{equation}
In particular, setting $r=d_1(x)$ in (\ref{equa22}), we obviously get
\begin{equation}\label{equa23}
d_1^2(x)y[d_2(x),x]\in P\;\;\; \mbox{for all}\;\; x,y\in R.
\end{equation}
From (\ref{equa23}), it follows that for each $x\in R$, $d_1^2(x)\in P$ or $[d_2(x),x]\in P$. Now fix $x\in R$ and consider the identity $d_1^2(x)\in P$. By (\ref{equa6}) we must show that $\overline{[d_1(x),y]}\in Z(R/P)$ which yields $\overline{d_1(x)}\in Z(R/P)$. Thus, since $d_1(x)d_2(x)\in P$, we find that $d_1(x)\in P$ or $d_2(x)\in P$. Consequently, in both cases we have $[d_2(x),x]\in P$ for all $x\in R$. Therefore, linearizing the last expression, we obviously get
\begin{equation*}
[d_2(x),y]+[d_2(y),x]\in P\;\;\; \mbox{for all}\;\; x,y\in R.
\end{equation*}
Writing $yx$ for $y$ in the above relation, then the result obtained gives that
\begin{equation}
[y,x]Rd_2(x)\subseteq P\;\;\; \mbox{for all}\;\; x,y\in R.
\end{equation}
Once again applying the primeness of $P$, it follows that $[y,x]\in P$ for all $x,y\in R$ or $d_2(R)\subseteq P$. Returning to the hypothesis, we find that $\overline{[R,R]}\subseteq Z(R/P)$. Clearly, we conclude that $R/P$ is a commutative integral domain. \\
$(2) \implies (3)$ Assume that
\begin{equation}\label{eeqq1}
\overline{d_1(x)d_2(x^*)-x\circ x^*}\in Z(R/P)~~\mbox{for all}~~~~x\in R.
\end{equation}
A linearization of (\ref{eeqq1}) gives
\begin{equation*}
\overline{d_1(x)d_2(y^*)+d_1(y)d_2(x^*)-x\circ y^*-y\circ x^*}\in Z(R/P)~~\mbox{for all}~~~~x,y\in R.
\end{equation*}
This can be rewritten as
\begin{equation}\label{eeqq2}
\overline{d_1(x)d_2(y)-x\circ y+d_1(y^*)d_2(x^*)-y^*\circ x^*}\in Z(R/P)~~\mbox{for all}~~~~x,y\in R.
\end{equation}
Substituting $yh$ for $y$ in (\ref{eeqq2}), where $h\in Z(R)\cap H(R)\setminus \{0\}$, we obtain
\begin{equation}\label{eeqq3}
\overline{d_1(x)yd_2(h)+d_1(h)y^*d_2(x^*)}\in Z(R/P)~~\mbox{for all}~~~~x,y\in R.
\end{equation}
Replacing $y$ by $yk$ in (\ref{eeqq3}), where $k\in Z(R)\cap S(R)$, we find that
\begin{equation}\label{eeqq4}
\overline{d_1(x)yd_2(h)-d_1(h)y^*d_2(x^*)}\in Z(R/P)~~\mbox{for all}~~~~x,y\in R.
\end{equation}
Invoking (\ref{eeqq3}) together with (\ref{eeqq4}), one can easily see that
\begin{equation}\label{eeqq5}
2\overline{d_1(x)yd_2(h)}\in Z(R/P)~~\mbox{for all}~~~~x,y\in R.
\end{equation}
Writing $k$ instead of $y$ in (\ref{eeqq5}), where $k\in Z(R)\cap S(R)$, and using the $2$-torsion freeness hypothesis, we arrive at
\begin{equation*}
\overline{d_1(x)d_2(h)}\in Z(R/P)~~\mbox{for all}~~~~x\in R.
\end{equation*}
Therefore, we get $\overline{d_1(R)}\subseteq Z(R/P)$ or $d_2(h)\in P$ for all $h\in Z(R)\cap H(R)$. In the second case, relation (\ref{eeqq4}) reduces to $\overline{d_1(h)d_2(x)}\in Z(R/P)$. Thus, we have $d_2(R)\subseteq Z(R/P)$ or $d_1(h)\in P$ for all $h\in Z(R)\cap H(R)$.\\
Analogously, if $d_1(Z(R)\cap S(R))\subseteq P$ and $d_2(Z(R)\cap S(R))\subseteq P$, then putting $k$ instead of $y$ in (\ref{eeqq2}), where $k\in Z(R)\cap S(R)$, we find that
\begin{equation}\label{eeqq6}
\overline{-x+x^*}\in Z(R/P)~~\mbox{for all}~~~~x\in R.
\end{equation}
Again, writing $k$ for $y$ in (\ref{eeqq6}), where $k\in Z(R)\cap S(R)$, one can obviously get
\begin{equation}\label{eeqq7}
\overline{-x-x^*}\in Z(R/P)~~\mbox{for all}~~~~x\in R.
\end{equation}
Combining (\ref{eeqq6}) with (\ref{eeqq7}), we conclude that $\overline{x}\in Z(R/P)$ for all $x\in R$ and this assures that $R/P$ is an integral domain.\\
Now if $\overline{d_1(R)}\subseteq Z(R/P)$ or $\overline{d_2(R)}\subseteq Z(R/P)$, then by adopting a similar approach of the one used  in the preceding Theorem, we get the required result.
\end{proof}

\begin{cor}\label{cor9}
Let $(R,\ast)$ be a ring with involution of $P$-second kind where $P$ is a prime ideal of $R$ such that $char(R/P)\neq 2$. If $R$ admits a derivation $d$ of $R$, then the following assertions are equivalent:\\
{\rm (1)} $\overline{d(x)d(x^*)\pm [x,x^*]}\in Z(R/P)$ for all $x\in R$;\\
{\rm (2)} $\overline{d(x)d(x^*)\pm x\circ x^*}\in Z(R/P)$ for all $x\in R$;\\
{\rm (3)} $R/P$ is an integral domain.
\end{cor}

\noindent
As an application of Theorem \ref{thm4}, we get some commutativity criteria for prime rings.

\begin{cor}\label{cor10}
Let $R$ be a 2-torsion free prime ring with involution $*$ of the second kind. If $R$ admits two derivations $d_1$ and $d_2$, then the following assertions are equivalent:\\
{\rm (1)} $d_1(x)d_2(x^*)\pm [x,x^*]\in Z(R)$ for all $x\in R$;\\
{\rm (2)} $d_{1}(x)d_{2}(x^*)\pm x\circ x^*\in Z(R)$ for all $x\in R$;\\
{\rm (3)} $R$ is an integral domain.
\end{cor}

%================================

\begin{thm}\label{thm5}
Let $(R,\ast)$ be a ring with involution of $P$-second kind where $P$ is a prime ideal of $R$ such that $char(R/P)\neq 2$. If $R$ admits two derivations $d_1$ and $d_2,$ then the following assertions are equivalent:\\
{\rm (1)} $\overline{d_1(x^*x)-d_2(xx^*)-xx^*}\in Z(R/P)$ for all $x\in R$;\\
{\rm (2)} $R/P$ is an integral domain.
\end{thm}

\begin{proof}
From the hypothesis, we have
\begin{equation}\label{eqt1}
\overline{d_1(x^*x)-d_2(xx^*)-xx^*}\in Z(R/P)~~\mbox{for all}~~~~x\in R.
\end{equation}
Linearizing (\ref{eqt1}), we get
\begin{equation}\label{eqt2}
\overline{d_1(xy)-d_2(yx)-yx+d_1(y^*x^*)-d_2(x^*y^*)-x^*y^*}\in Z(R/P)~~\mbox{for all}~~~~x,y\in R.
\end{equation}
Replacing $y$ by $yh$ in (\ref{eqt2}), where $h\in H(R)\cap Z(R)$, we find that
\begin{equation}\label{eqt3}
\overline{(xy+y^*x^*)d_1(h)-(yx+x^*y^*)d_2(h)}\in Z(R/P)~~\mbox{for all}~~~~x,y\in R.
\end{equation}
Substituting $yk$ for $y$ in (\ref{eqt3}), where $k\in S(R)\cap Z(R)\setminus\{0\}$, we obtain
\begin{equation}\label{eqt4}
\overline{(xy-y^*x^*)d_1(h)-(yx-x^*y^*)d_2(h)}\in Z(R/P)~~\mbox{for all}~~~~x,y\in R.
\end{equation}
Combining the two last relations, we arrive at
\begin{equation}\label{eqt5}
\overline{xyd_1(h)-yxd_2(h)}\in Z(R/P)~~\mbox{for all}~~~~x,y\in R.
\end{equation}
As a special case of (\ref{eqt5}), when we put $y=k$, we may write
\begin{equation}\label{eqt6}
\overline{x(d_1(h)-d_2(h))}\in Z(R/P)~~\mbox{for all}~~~~x,y\in R.
\end{equation}
Replacing $h$ by $k^2$ in (\ref{eqt3}) and using the primeness of $P$, we can see that
\begin{equation*}
\overline{x}\in Z(R/P)~~\mbox{for all}~~x\in R\quad \mbox{or}\quad d_1(k)-d_2(k)\in P.
\end{equation*}
By view of (\ref{eqt5}), it results
\begin{equation*}
\overline{[x,y]d_2(k)}\in Z(R/P)~~\mbox{for all}~~~~x,y\in R.
\end{equation*}
Therefore, $d_2(Z(R)\cap S(R))\subset P$ or $R/P$ is commutative. Arguing as above, we find that $d_1(Z(R)\cap S(R))\subset P$. Now let $y=k$ in (\ref{eqt2}), it follows that
\begin{equation}\label{eqt7}
\overline{d_1(x)-d_2(x)-x}\in Z(R/P)~~\mbox{for all}~~~~x\in R.
\end{equation}
In such a way, (\ref{eqt7}) leads to $[D(x),x]\in P$ for all $x\in R$. where $D:x\mapsto d_1(x)-d_2(x)$ is a derivation of $R$. According to (\cite{AMT}, Lemma 2.1), we arrive at $D(R)\subseteq P$ or $R/P$ is an integral domain.\\
If $D(R)\subseteq P$, then our assumption becomes
\begin{equation}\label{eqt8}
\overline{d_2[x,x^*]-xx^*}\in Z(R/P)~~\mbox{for all}~~~~x\in R.
\end{equation}
Writing $x^*$ for $x$ in (\ref{eqt8}) and subtracting it from the above relation, we conclude that
\begin{equation*}
\overline{[x,x^*]}\in Z(R/P)~~\mbox{for all}~~~~x\in R.
\end{equation*}
A Linearization of this relation, yields
\begin{equation}\label{eqt9}
\overline{[x,y^*]+[y,x^*]}\in Z(R/P)~~\mbox{for all}~~~~x,y\in R.
\end{equation}
By an appropriate expansion
\begin{equation}\label{eqt10}
\overline{-[x,y^*]+[y,x^*]}\in Z(R/P)~~\mbox{for all}~~~~x,y\in R.
\end{equation}
Combining (\ref{eqt9}) with (\ref{eqt10}), we arrive at
\begin{equation}\label{eqt11}
\overline{[y,x^*]}\in Z(R/P)~~\mbox{for all}~~~~x,y\in R.
\end{equation}
Putting $yx$ for $x$ in (\ref{eqt11}), one can verify that
\begin{equation}\label{eqt12}
[x,x^*]\in P~~\mbox{for all}~~~~x\in R.
\end{equation}
Applying Lemma \ref{lem}, we get the required result.
\end{proof}

\begin{cor}\label{cor11}
Let $(R,\ast)$ be a ring with involution of $P$-second kind where $P$ is a prime ideal of $R$ such that $char(R/P)\neq 2$. If $R$ admits a derivation $d$, then the following assertions are equivalent:\\
{\rm (1)} $\overline{d[x,x^*]-xx^*}\in Z(R/P)$ for all $x\in R$;\\
{\rm (2)} $\overline{d(x\circ x^*)-xx^*}\in Z(R)$ for all $x\in R$;\\
{\rm (2)} $R/P$ is an integral domain.
\end{cor}

\begin{cor}\label{cor12}
Let $R$ be a 2-torsion free prime ring with involution $*$ of the second kind. If $R$ admits two derivations $d_1$ and $d_2$, then the following assertions are equivalent:\\
{\rm (1)} $d_1(x^*x)-d_2(xx^*)-xx^*\in Z(R)$ for all $x\in R$;\\
{\rm (2)} $R$ is an integral domain.
\end{cor}

\noindent The following example shows that the "\verb"primeness"" hypothesis in Theorem \ref{thm3} is not superfluous.\\

\noindent \textbf{Example 1.} Consider the ring $\mathcal{R} = M_2(\mathbb{C})\times \mathbb{Z}[X]$ and $\mathcal{P} = 0\times P$, where $P$ any proper ideal of $\mathbb{Z}[X]$. Let $\sigma$ be the involution of $P$-second kind defined on $\mathcal{R}$ by $$\left(\begin{pmatrix}
a & b\\
c & d
\end{pmatrix},P(X)\right)^\sigma=\left(\begin{pmatrix}
\overline{a} & \overline{c}\\
\overline{b} & \overline{d}
\end{pmatrix},P(X)\right).$$
If we set $d_1(M,P(X)) = \left(\begin{pmatrix}
0 & -b\\
-c & 0
\end{pmatrix},0\right)$ and $d_2(M,P(X))=(0,P'(X)),$ then an easy computation  leads to $\overline{d_1(x)d_2(x^*)}\in Z(\mathcal{R}/\mathcal{P})$ for all $x\in \mathcal{R}$; but none of the conclusions of Theorem \ref{thm3} is satisfied.
\\

\noindent
{\bf \bfseries Acknowledgment. }\\
The authors are deeply indebted to the learned referees for their careful reading of the manuscript and constructive
comments.\\

\noindent
{\bf \bfseries Declaration. }\\
There is no funding source.\\

\noindent
{\bf \bfseries Conflicts of Interest. }\\
The authors declare no conflicts of interest.\\

%==================================

\end{document}